\documentclass[12pt]{amsart}
\usepackage{latexsym,amssymb,amsmath}
\usepackage{amsmath,amsfonts}
\usepackage{epsfig}
\usepackage{graphicx}
\usepackage{verbatim}

\newtheorem{theorem}{Theorem}[section]
\newtheorem{lemma}[theorem]{Lemma}
\newtheorem{proposition}[theorem]{Proposition}
\newtheorem{cor}[theorem]{Corollary}
\newtheorem{definition}[theorem]{Definition}
\newtheorem{remark}[theorem]{Remark}

\allowdisplaybreaks

\newcommand{\ra}{\rightarrow}

\def\nn{\nonumber}
\def\T{\mathbb T}

\def\C{{\mathbb C}}
\def\N{{\mathbb N}}

\def\R{{\mathbb R}}
\def\Z{{\mathbb Z}}
\def\la{\langle}
\def\ra{\rangle}

\def\les{\lesssim}
\def\1{{\bf 1}}

\def\eqnn{\begin{eqnarray*}}
\def\eeqnn{\end{eqnarray*}}
\def\eqn{\begin{eqnarray}}
\def\eeqn{\end{eqnarray}}
\newcommand{\nc}{\newcommand}
\nc{\be}{\begin{equation}}
\nc{\ee}{\end{equation}}
\nc{\ba}{\begin{eqnarray}}
\nc{\ea}{\end{eqnarray}}
\nc{\eps}{\epsilon}
\def\prf{\begin{proof}}
\def\endprf{\end{proof}}
\begin{document}

\title[Fractal solutions of dispersive PDE]{Fractal solutions of linear and nonlinear dispersive partial differential equations}

\author{V. Chousionis}
\address{Vasilis Chousionis, Department of Mathematics and Statistics \\ University of Helsinki \\ P. O. Box 68 \\ FI-00014, Finland}
\email{vasileios.chousionis@helsinki.fi}
\author{M. B. Erdo\u gan}
\address{M. Burak Erdo\u gan, Department of Mathematics \\ University of Illinois \\ 1409
  West Green St. \\ Urbana, IL, 61801}
\email{berdogan@math.uiuc.edu}
\author{N. Tzirakis}
\address{Nikolaos Tzirakis, Department of Mathematics \\ University of Illinois \\ 1409
  West Green St. \\ Urbana, IL, 61801}
\email{tzirakis@math.uiuc.edu}

\thanks{V.C. is supported  by the Academy of Finland Grant SA 267047. M.B.E.  is partially supported by the NSF grant DMS-1201872}

\date{}

\begin{abstract}
In this paper we  study fractal solutions of linear and nonlinear dispersive PDE on the torus. In the first part we answer some open questions  on the fractal solutions of linear Schr\"odinger equation and equations with higher order dispersion. We also discuss applications to their nonlinear counterparts like the cubic Schr\"odinger equation (NLS) and the Korteweg-de Vries equation (KdV). 

In the second part,   we study fractal solutions of the vortex filament equation and the associated Schr\"odinger map  equation (SM). In particular, we construct
global strong solutions of the SM in $H^s$ for $s>\frac32$ for which the evolution of the curvature   is given by a periodic nonlinear Schr\"odinger evolution. We also construct unique weak solutions in the energy level.  Our analysis follows the frame construction of Chang {\em et  al.} \cite{csu} and Nahmod {\em et al.} \cite{nsvz}.
\end{abstract}

\maketitle

\section{Introduction}

In this paper we continue the study of  fractal solutions of linear and nonlinear dispersive PDE on the torus that was initiated in \cite{erdtzi1}. We present dispersive quantization effects that were observed numerically for discontinuous initial data, \cite{chenolv}, in a large class of dispersive PDE, and in certain geometric equations \cite{hozvega}. A physical manifestation of these phenomena started with an optical experiment of Talbot \cite{talbot} which today is referred in the literature as the Talbot effect.
Berry with his collaborators (see, e.g., \cite{mber, berklei,berlew, bermar})   studied the Talbot effect in a series of papers.
In particular, in \cite{berklei},  Berry and Klein  used the linear Schr\"odinger evolution to model the Talbot effect. Also in  \cite{mber}, Berry  conjectured that  for the $n-$dimensional linear Schr\"odinger equation confined in a box
 the graphs of the imaginary part $\Im u(x,t)$, the real part $\Re{u(x,t)}$ and the density $|u(x,t)|^2$ of the solution are fractal sets with dimension $D=n+\frac{1}{2}$ for most irrational times. 
We should also note that in \cite{ZWZX} the Talbot effect was observed experimentally in a nonlinear setting.    

The first mathematically rigorous work in this area appears to be due to Oskolkov. In \cite{osk}, he studied  a large class of linear dispersive equations with bounded variation initial data.    In the case of the linear Schr\"odinger equation, he proved that  at irrational times the solution is a continuous function  of $x$ and at rational times it is a bounded function with at most countably many discontinuities.
The idea that the profile of linear dispersive equations depends on the algebraic properties of time was further investigated by Kapitanski-Rodniaski  \cite{rodkap}, Rodnianski  \cite{rod}, and Taylor  \cite{mtay}.  
Dispersive quantization results have also been observed on higher dimensional spheres and  tori \cite{mtay2}. In \cite{mtay}, Taylor  noted that the quantization implies the $L^p$ boundedness of the multiplier $e^{it\Delta}$ for rational values of $\frac{t}{2\pi}$. It is  known that, \cite{mtay}, the propagator is unbounded in $L^p$ for $p\neq 2$ and $\frac{t}{2\pi}$ irrational.  This can be considered as another manifestation of the Talbot effect.

In \cite{erdtzi1}, the second and third authors investigated the Talbot effect for  cubic  nonlinear  Schr\"odinger equation (NLS) with periodic boundary conditions.
The goal was to extend Oskolkov's and Rodnianski's results for bounded variation data to the NLS evolution, and  provide rigorous confirmation of some numerical observations in \cite{olv,chenolv,chenolv1}. We recall the main theorem of \cite{erdtzi1}:

\noindent
{\bf Theorem A.} \cite{erdtzi1}
{\it Consider the nonlinear Schr\"odinger equation on the torus
\begin{align*}
&iu_t+u_{xx}+|u|^2u=0,\,\,\,\,\,\, t\in \R,\,\,\,x\in\T=\R/2\pi\Z,\\
&u(x,0)=g(x).
\end{align*}
Assuming that $g$ is of bounded variation, we have\\
i)  $u(x,t)$ is a continuous function of $x$ if $\frac{t}{2\pi}$ is an irrational number. For rational values of $\frac{t}{2\pi}$, the solution is a bounded function with at most countably many discontinuities.  Moreover, if $g$ is also continuous  then $u\in C^0_tC^0_x$.\\
ii) If in addition $g\not \in \bigcup_{\epsilon>0}H^{\frac12+\epsilon}$,
then for almost all times    either the real part or the imaginary part of the graph of
$u(\cdot,t)$ has upper Minkowski dimension $ \frac32$.}

\vspace{3mm}

We note that the simulations in \cite{olv,chenolv,chenolv1} were performed in the case when $g$ is a step function, and that Theorem A applies in that particular case. The proof of Theorem A relies on a smoothing estimate for NLS stating that the nonlinear Duhamel part of the evolution is smoother than the linear part by almost half a derivative. For bounded variation data, this immediately yields the upper bound on the dimension of the curve.  The lower bound is obtained by combining the smoothing estimate with Rodnianski's result in \cite{rod},
and an observation from \cite{DJ} connecting smoothness and geometric dimension. We remark that the first part of 
Theorem A was observed in \cite{erdtzi} in the case of KdV equation.

In this article we first show how one can obtain the same theorem as above for both the real part and the imaginary part of the graph of
$e^{it\partial_{xx}} g $.  Then we prove that the linear Schr\"odinger evolution gives rise to fractal curves even for smoother data. In particular we show that if the initial data is of bounded variation but do not belong in $\bigcup_{\epsilon>0}H^{r_0+\epsilon}$ for some $r_0\in[\frac12,\frac34) $, then for almost all $t$ both  the real part and the imaginary part of the graph of
$e^{it\partial_{xx}} g $ have upper Minkowski dimension $D\in [\frac52-2r_0,\frac32]$.   Notice that  for $r_0\in[\frac12,\frac34) $, $1<\frac52-2r_0\leq \frac32$. These results apply to NLS evolution as in Theorem A above.
Our next theorem addresses Berry's conjecture regarding the fractal dimension of the density of the linear Schr\"odinger equation $|e^{it\partial_{xx}} g|^2$. Although we are unable to prove the statement for general bounded variation initial data we nevertheless prove the dimension statement for step function data with jumps only at rational points. Our result confirm the numerical simulations that have appeared in the literature.  We also note that our theorem implies the same statement for the absolute value of the solution $|e^{it\partial_{xx}} g|$.

The numerical simulations in Olver \cite{olv}, and  Chen and Olver  \cite{chenolv,chenolv1}  validated the Talbot effect for a large class of dispersive equations, both linear and nonlinear. In the case of polynomial dispersion, they numerically confirmed the rational/irrational dichotomy discussed above. This behavior persists for both integrable and nonintegrable systems. Our next theorem addresses exactly this problem. We consider for any $k \geq 3$ the following linear dispersive class of PDE:
\begin{align*}
&iu_t+(-i\partial_x)^ku=0,\,\,\,\,\,\, t\in \R,\,\,\,x\in\T=\R/2\pi\Z,\\
&u(x,0)=g(x)\in BV.
\end{align*}
We prove that for almost all $t$ both  the real part and the imaginary part of the graph of
$e^{it (-i \partial_x)^k}g $ is a fractal curve with upper Minkowski dimension $D\in [1+2^{1-k},2-2^{1-k}]$. In particular the  upper Minkowski dimension $D\in [\frac54,\frac74]$ for almost every $t$ in the case of the Airy equation ($k=3$). The dimension bounds are also valid for the KdV evolution, see below.

An important question that the authors raised in \cite{chenolv,chenolv1} is the appearance of such phenomena in the case of nonpolynomial dispersion relations. Their numerics demonstrate that the large wave number asymptotics of the dispersion relation plays the dominant role governing the qualitative features of the solutions. We will address these phenomena in future work. Here we just want to note that smoothing estimates for fractional Schr\"odinger equations have already been proved in \cite{det}.

In the second part of our paper we investigate fractal solutions of the vortex filament equation (VFE):
\begin{equation}\label{vfe}
\gamma_{t}=\gamma_{x}\times \gamma_{xx}, 
\end{equation}
where $\gamma:\R\times \Bbb K\rightarrow \R^3$, satisfies $|\gamma_x|=1$, i.e.  $x$ is an arc-length parameter. Here  the field $\Bbb K$ can be either $\R$ or the torus $\T$. This equation was first  discovered by Da Rios, \cite{dr}, and it models the dynamics of an isolated thin vortex embedded in a homogeneous, incompressible, inviscid fluid. Da Rios wanted to study the influence that the localized vorticity has on the local and global behavior of the vortex, \cite{dr}. In this model the velocity of the vortex is proportional to its local curvature (thus smaller rings move faster). 

VFE is connected to the cubic NLS through Hasimoto's transformation \cite{hasi}. Hasimoto coupled the curvature and torsion of the filament into one complex variable and derived a nonlinear Schr\"odinger equation that governs the dynamics of the vortex filament. We should note that VFE in $H^s$ level formally corresponds to NLS in $H^{s-2}$ level.

Recently in  \cite{hozvega},   De la Hoz and Vega considered solutions of the VFE with initial data a regular planar polygon. Formally, at the NLS level this corresponds to initial data represented as a sum of delta functions with appropriate weights. Using algebraic techniques, supported
by numerical simulations, they demonstrated that $u$ is also a
polygonal curve at any rational time. They also studied numerically the fractal behavior for irrational times. For example their simulations demonstrate that the stereographic projection of the unit tangent vector at an irrational time is a fractal like curve. 

Since NLS is known to be ill-posed below the $L^2$ level, it appears that a rigorous justification of the observations  in \cite{hozvega} is out of reach. Instead, in this paper we prove that the VFE has solutions with some fractal behavior even when the data is much smoother then a polygonal curve.  
For example for $C^1(\T)$ initial data which is a planar curve of piecewise constant curvature, we prove that the curvature vector $\gamma_{xx}=\kappa N$ of the filament   has fractal coordinates with respect to a frame, see the discussion below.

To construct solutions of VFE we consider the Schr\"odinger map equation (SM)
\begin{equation}\label{sm}
u_{t}=u\times u_{xx},
\end{equation}
where $u=\gamma_x: \R \times \Bbb K\rightarrow S^2$, where $S^2$ is the unit sphere in $\R^3$. The SM  has a long history and can be derived as a model on the effects of a magnetic field on ferromagnetic materials. A ferromagnetic material can be viewed as a collection of atoms each with a well defined magnetic moment which interact with its neighbors. The dynamics of the magnetic moment (spin) are governed by 
$$u_t=u\times F$$
where $F=-\frac{\delta E}{\delta u}$. If the ferromagnetic energy is $E(u)=\frac{1}{2}\int_{\Omega} |\nabla u|^2$ then the equation takes the form $u_t=u\times \Delta u$.

In the case of $\Bbb K=\T$, since $u$ is the derivative of the curve with respect to arc-length, it has mean zero:
$$\int_{\Bbb K}u(x)dx=0.$$
In addition, since $|u|=1$ for each $x$, $\|u\|_{L^2(\T)}=2\pi$ for all times. Moreover, noting that
$$\partial_{t} \int_{\T}u_x \cdot u_x\ dx=2\int_{\T}u_{tx} \cdot u_x\ dx=-2\int_{\T}u_t\cdot u_{xx}\ dx=-2\int_{\T}(u\times u_{xx})\cdot u_{xx}\ dx=0,$$
we see that the smooth solutions of the SM have constant $H^1(\T)$ norm:
$$\|u(t)\|_{H^{1}(\T)}=\|u_0\|_{H^{1}(\T)}.$$

Note that the problem is $H^{\frac{n}{2}}$ critical and thus in 1d it is energy sub-critical.
Many results have been established for the SM   from $\R^n$ to $S^2$.  We cannot summarize all of them here but we should mention  a recent result, \cite{bikt}, proving small data global well-posedness in the critical space for any $n\geq 2$. For the 1d case when the base is $\R$ or the torus $\T$ and the target is the sphere the best result is the global well-posedness in $H^2$. For the real line, existence was proved in \cite{ding} for large $H^2(\R)$ data, while uniqueness was shown in $H^3(\R)$. Uniqueness in $H^2(\R)$ was proved by Chang, Shatah and Uhlenbeck \cite{csu}, also see \cite{nsvz} for further clarifications. The strategy used in  \cite{csu} and \cite{nsvz} is to write the derivative of the  solution in a special orthonormal frame in the tangent plane in which the equation turns out to be of NLS type.
For the global well-posedness in $H^2(\T)$, see \cite{rrs}, which also addresses SM from Riemanian manifolds to K\"ahler manifolds.

Since the cubic NLS is well-posed  in $L^2$, heuristically the SM should be well-posed in $H^1$. However, this is still open, since translating the estimates available for the NLS to   the SM is nontrivial.  
Recall that for the periodic cubic NLS we observed fractal  solutions for data in $H^s$ level, $s<\frac34$. Thus,    in principle, we expect to have fractal solutions for the SM at the $H^s$ level for $s<\frac74$. To do that we restrict ourselves to mean zero and identity holonomy initial data on the torus. By identity holonomy we mean that the parallel transport around the curve $u_0(\T)$ is the identity map on the tangent space $T_{u_0(0)}S^2$.  By the Gauss-Bonnet theorem, for smooth curves, this is equivalent to the condition that the area enclosed by the curve counting multiplicities in $S^2$ is an integer multiple of $2\pi$. In particular, planar initial data  are always identity holonomy. 

To study the well posedness of the SM we use the frame construction in \cite{csu} and \cite{nsvz}, which converts SM to a simpler system of ODE. We note that for smooth and identity holonomy data the frame is also $2\pi$-periodic and the coefficients of $u_x$ in the frame evolves according to NLS on $\T$.  With the help of the system and the conservation laws of the SM we obtain  a unique solution  as a strong limit of smooth solutions and prove that the SM is globally well posed in $H^s(\T)$ for any $s>\frac32$. Moreover the coefficients of the curvature vector $ u_x $ with respect to the frame are  given by the real and imaginary parts of a function $q\in L^\infty_tH^{s-1}_x$ which solves NLS on $\T$.

 We also address the problem of the uniqueness of the weak solutions of the SM. 
For weak solutions of SM see the paper \cite{JS} and the references therein.
The construction of weak solutions in the energy space was proved in \cite{SSB}, also see the discussion in \cite{JS}.  
Certain uniqueness statements (under assumptions on NLS evolution on $\R$) were obtained in \cite{nsvz}. We obtain unique weak solutions in $H^s(\T)$ for $s\geq 1$ and for identity holonomy and mean zero data. These solutions  are weakly continuous in $H^s$ and continuous in $H^r$ for $r<s$. Moreover, for $s>1$ the curvature vector $u_x$ is given by a NLS evolution in $H^{s-1}$ level as we noted above.

\section{Notation}
To avoid the use of multiple constants, we  write $A \lesssim B$ to denote that there is an absolute  constant $C$ such that $A\leq CB$. We define $\langle \cdot\rangle =1+|\cdot|$. 

We define the Fourier sequence of a $2\pi$-periodic $L^2$ function $u$ as
$$\hat{u}(k)=u_k=\frac1{2\pi}\int_0^{2\pi} u(x) e^{-ikx} dx, \,\,\,k\in \mathbb Z.$$
With this normalization we have
$$u(x)=\sum_ke^{ikx}u_k,\,\,\text{ and } (uv)_k=u_k*v_k=\sum_{m+n=k} u_nv_m.$$
We will also use the notation:
$$
P_0u=u_0=\frac1{2\pi} \int_\T u.
$$

Note that for a mean-zero $L^2$ function $u$, $\|u\|_{H^{s}}=\|u\|_{\dot H^{s}}\approx \|\widehat u(k) |k|^{s}\|_{\ell^2}$.

Similarly, for $u:\T\to \R^3$, the Fourier coefficients are $u_k=(u_{1,k},u_{2,k},u_{3,k})$, and 
$$\|u\|_{L^2}=\Big(\int_\T u\cdot u \, dx \Big)^{1/2},\,\,\,\,\,\,\,\,\,\,\|u\|_{H^1}= \|\partial_x u\|_{L^2} +\|u\|_{L^2}.$$
For general $s$ we have
$$\|u\|_{H^s}^2\approx \sum_k \la k\ra^{2s} u_k\cdot u_k.$$

The upper Minkowski (also known as fractal) dimension, $\overline{\text{dim}}(E)$, of a bounded set $E$ is given by $$\limsup_{\epsilon\to 0}\frac{\log({\mathcal N}(E,\epsilon))}{\log(\frac1\epsilon)},
$$
where ${\mathcal N}(E,\epsilon)$ is the minimum number of $\epsilon$--balls required to cover $E$.

Finally, by local and global well-posedness we mean the following. 
\begin{definition} We say the equation  is locally well-posed in $H^s$, if there exist a time $T_{LWP}=T_{LWP}(\|u_0\|_{H^s})$ such that the solution exists and is unique in $X_{T_{LWP}}\subset C([0,T_{LWP}),H^s)$ and depends continuously on the initial data. We say that the the equation is globally well-posed when $T_{LWP}$ can be taken arbitrarily large.
\end{definition}

\section{Fractal solutions of dispersive PDE on $\T$}
We will start with the linear Schr\"odinger evolution $e^{it\partial_{xx}}g$. The following theorem is a variant of  the results in \cite{osk} and \cite{rod}: 
\begin{theorem}\label{osk_rod}
Let  $g:\T \to \C $ be of bounded variation. Then $e^{it\partial_{xx}}g$ is a continuous function of $x$ for almost every $t$. Moreover if in addition  $g\not \in \bigcup_{\epsilon>0}H^{r_0+\epsilon}$ for some $r_0\in[\frac12,\frac34) $, then for almost all $t$ both  the real part and the imaginary part of the graph of
$e^{it\partial_{xx}} g $ have upper Minkowski dimension $D\in [\frac52-2r_0,\frac32]$. In particular, for $r_0=\frac12$, $D=\frac32$.
\end{theorem}
Before we prove this theorem we need the following lemma. 
\begin{lemma}\label{real_imag}
Let $g:\T \to \C $ be of bounded variation. Assume that $r_0:=\sup\{s:g\in H^s\}\in [\frac12,1)$  Then for almost every $t$, both the real and imaginary parts of $e^{it\partial_{xx}}g$ do not belong to $H^r$ for $r>r_0$.
\end{lemma}
\begin{proof}
We prove this for the real part, the same argument works for the imaginary part. Also, we can assume that $r < \frac{r_0+1}2$.
It suffices to prove that for a subsequence $\{K_n\}$ of $\N$,
$$
\sum_{k=1}^{K_n} k^{2r} |e^{-itk^2}\widehat{g}(k)+e^{ itk^2}\overline{\widehat{g}(-k)}|^2   \to \infty \text{ for almost every } t.
$$
We have
$$
\sum_{k=1}^{K } k^{2r} |e^{-itk^2}\widehat{g}(k)+e^{ itk^2}\overline{\widehat{g}(-k)}|^2   = \sum_{k=1}^{K } k^{2r} (|\widehat{g}(k)|^2 + |\widehat{g}(-k)|^2 )+ 2\Re \big(\sum_{k=1}^{K } k^{2r} e^{-2itk^2}  \widehat{g}(k)\widehat{g}(-k) \big).
$$
Since the first sum diverges as $K\to \infty$, it suffices to prove that the second sum converges almost everywhere after passing to a subsequence. As such it suffices to prove that it converges in $L^2(\T)$, which immediately follows from  Plancherel as
$$
\sum_{k=1}^\infty k^{4r} |\widehat{g}(k)|^2|\widehat{g}(-k)|^2 \lesssim \sup_k k^{4r-2r_0-2+}  \|g\|_{H^{r_0-}}^2   <\infty.
$$
In the last two inequalities we used the bound $|\widehat{g}(k)|\les |k|^{-1}$ and that  $r < \frac{r_0+1}2$.
\end{proof}
\begin{proof}[Proof of Theorem~\ref{osk_rod}]
Consider 
$$
H_{N,t}(x)=\sum_{0<|n|\leq N} \frac{e^{-itn^2+inx}}{n}=\sum_{n=1}^N  \frac{e^{-itn^2+inx}-e^{-itn^2-inx}}{n}.
$$
Let 
$$
T_{N,t}(x)=\frac1N \sum_{n=1}^N  \big[e^{-itn^2+inx}-e^{-itn^2-inx}\big]
$$
By the summation by parts formula,
$$
\sum_{n=1}^N f_n(g_{n+1}-g_n)=f_{N+1}g_{N+1}-f_1g_1-\sum_{n=1}^{N}g_{n+1}(f_{n+1}-f_n),
$$
with $g_n=(n-1)T_{n-1,t}$ and $f_n=\frac1n$, we have
\begin{align}
\label{hntn}
H_{N,t}(x)=\frac{N}{N+1} T_{N,t}(x) + \sum_{n=1}^N  \frac{T_{n,t}(x)}{n+1}=T_{N,t}(x)+ \sum_{n=1}^{N-1}  \frac{T_{n,t}(x)}{n+1}.
\end{align}

We will use the following well-known results from number theory:
\begin{theorem}\label{weyl}\cite{montgomery} Let $t$ satisfy
\be\label{dirichlet}
\big|\frac{t}{2\pi}-\frac{a}{q}\big|\leq \frac1{q^2} 
\ee
for some integers $a$ and $q$, then
$$
\sup_x \big|\sum_{n=1}^N e^{-itn^2+inx}\big| \lesssim \frac{N}{\sqrt{q}}+\sqrt{N\log q} + \sqrt{q\log q}.
$$
\end{theorem}
Recall that Dirichlet theorem implies that for every irrational $\frac{t}{2\pi}$, the inequality \eqref{dirichlet} holds for infinitely many integers $a, q$.  Given irrational $\frac{t}{2\pi}$, let $\{q_k\} $ be the increasing sequence of positive $q$'s for which \eqref{dirichlet} holds for some $a$.  We need the following quantitative information on the sequence  $\{q_k\}$.

\begin{theorem}\cite{khin, lev}
 For almost every  $t$, we have $\lim_{k\to\infty} q_k^{1/k} = \gamma$, for some absolute constant $\gamma$ independent of $t$.
\end{theorem}

An immediate corollary of this theorem is the following:
\begin{cor}
For almost every $t$, and for any $\epsilon>0$, we have
$$q_{k+1}\leq q_k^{1+\epsilon} $$
for all sufficiently large $k$.
\end{cor}
This in turn implies that
\begin{cor}\label{khinlev}
For almost every    $t$, for any $\epsilon>0$, and for all sufficiently large $N$, there exists $q\in [N,N^{1+\epsilon}]$ so that \eqref{dirichlet} holds for $q$.
\end{cor}

Combining Theorem~\ref{weyl} and Corollary~\ref{khinlev}, we obtain
\begin{cor}\label{numtheorycor}
For almost every   $t$, and for any $\epsilon>0$, we have
$$
\sup_x \big|\sum_{n=1}^N e^{-itn^2+inx}\big| \lesssim N^{\frac12+\epsilon}
$$
for all  $N$.
\end{cor}
Using Corollary~\ref{numtheorycor} in \eqref{hntn},  we see that for almost every $t$  the sequence $H_{N,t}$ converges uniformly to a continuous function
$$
H_{t}(x)=\sum_{n\neq 0} \frac{e^{-itn^2+inx}}{n}.
$$
It also implies that for any $\epsilon>0$, and for any $j=1,2,...,$
$$
\Big\|\sum_{2^{j-1}\leq |n| <2^j}  \frac{e^{-itn^2+inx}}{n} \Big\|_{L^\infty_x}\les 2^{-j(\frac12-\epsilon)}.
$$

Therefore for almost every $t$
$$
H_t\in \bigcap_{\epsilon>0} B_{\infty,\infty}^{\frac12-\epsilon}(\T).
$$
Recall that   the Besov space $B^s_{p,\infty}$ is defined via the norm:
$$
\|f\|_{B^s_{p,\infty}}:=\sup_{j\geq 0}2^{sj} \|P_j f\|_{L^p},
$$
where $P_j$ is a Littlewood-Paley projection on to the frequencies $\approx 2^j$.

Now, given function $g$ of bounded variation, we write
$$
e^{it\partial_{xx}}g=\widehat{g}(0)+\sum_{n\neq 0} e^{-itn^2+inx} \widehat{g}(n).
$$
Note that
$$
\widehat{g}(n)=\frac1{2\pi}\int_\T e^{-iny} g(y)dy =\frac1{2\pi i n} \int_\T e^{-iny} dg(y),
$$
where $dg$ is the Lebesgue-Stieltjes measure associated with $g$. Therefore
$$
e^{it\partial_{xx}}g=\widehat{g}(0)+\lim_{N\to\infty} H_{N,t} * dg = \widehat{g}(0)+H_{t}* dg
$$
by the uniform convergence of the sequence $H_{N,t}$.  In particular, for almost every $t$,
\be\label{Binf}
e^{it\partial_{xx}} g \in \Big[\bigcap_{\epsilon>0} B_{\infty,\infty}^{\frac12-\epsilon}(\T)\Big] \bigcap C^0(\T).
\ee
Now in addition assume that $g\not\in H^r(\T)$ for any $r>r_0\geq \frac12$. This implies using Lemma~\ref{real_imag} that
$$
\Im e^{it\partial_{xx}} g ,\,\, \Re e^{it\partial_{xx}} g \not \in \bigcup_{\epsilon>0} B^{2r_0-\frac12+\epsilon}_{1,\infty}(\T),
$$
since
\be\label{interpol}
H^r(\T) \supset B_{1,\infty}^{r_1}(\T) \cap B_{\infty,\infty}^{r_2}(\T),
\ee
for $r_1+r_2>2r$.

The lower bound for the upper Minkowski dimension in   Theorem~\ref{osk_rod}  follows from the following theorem of Deliu and Jawerth \cite{DJ}.

\begin{theorem}\cite{DJ} The graph of a continuous function $f:\T\to\R$  has upper Minkowski dimension $D \geq 2-s$ provided that $f\not\in\bigcup_{\epsilon>0} B^{s+\epsilon}_{1,\infty}$.
\end{theorem}

We should now note that $C^\alpha(\T)$ coincides with $B^\alpha_{\infty,\infty}(\T)$, see, e.g., \cite{tri}, and that  if $f:\T\to\R$ is in $C^\alpha$, then
the graph of $f$ has upper Minkowski dimension $D\leq 2-\alpha$. Therefore, the graphs of $\Re(u)$ and $\Im(u)$ have dimension at most $\frac32$ for almost all $t$.

\end{proof}

The smoothing result in \cite{erdtzi1} and  Theorem~\ref{osk_rod} above  imply as in the proof of Theorem~A  the following:

\begin{theorem} 
 Consider the nonlinear Schr\"odinger equation on the torus
\begin{align*}
&iu_t+u_{xx}+|u|^2u=0,\,\,\,\,\,\, t\in \R,\,\,\,x\in\T=\R/2\pi\Z,\\
&u(x,0)=g(x),
\end{align*}
where  $g:\T \to \C $ is of bounded variation and $g\not \in \bigcup_{\epsilon>0}H^{r_0+\epsilon}$ for some $r_0\in[\frac12,\frac34) $, then for almost all $t$ both  the real part and the imaginary part of the graph of
$u $ have upper Minkowski dimension $D\in [\frac52-2r_0,\frac32]$.  
\end{theorem}

We now turn our attention to the problem of fractal dimension of the density function $|e^{it\partial_{xx}} g|^2$. This problem was left open in \cite{rod}.
\begin{theorem}\label{thm:density} 
Let $g$ be a nonconstant complex valued step function on $\T$ with jumps only at rational multiples of $\pi$.
Then for almost every $t$ the graphs of $|e^{it\partial_{xx}} g|^2$ and $|e^{it\partial_{xx}} g|$ have upper Minkowski dimension $\frac32$.
\end{theorem}
The upper bound follows immediately from the proof of Theorem~\ref{osk_rod} since $C^\alpha(\T)$ is an algebra. The lower bound for $|e^{it\partial_{xx}} g|$ follows from the lower bound for $|e^{it\partial_{xx}} g|^2$ since $|e^{it\partial_{xx}} g|$ is a continuous and hence bounded function for almost every $t$. The lower bound for $|e^{it\partial_{xx}} g|^2$ follows from the same proof above provided that we have
$$
|e^{it\partial_{xx}} g|^2 \in  \bigcap_{\epsilon>0} B_{\infty,\infty}^{\frac12-\epsilon}(\T) 
$$
and 
$$|e^{it\partial_{xx}} g|^2\not\in H^{\frac12}.$$
The former follows from \eqref{Binf} since $B^\alpha_{\infty,\infty}=C^\alpha$ is an algebra. 
For the latter we have:
\begin{proposition}
Let $g$ be a nonconstant complex valued step function on $\T$ with jumps only at rational multiples of $\pi$.
 Then for every irrational value of $\frac{t}{2\pi}$, we have $|e^{it\partial_{xx}} g|^2\not\in H^{\frac12}$.
\end{proposition}
\begin{proof} 
We can write (after a translation and adding a constant to $g$) $g=\sum_{\ell=1}^{L} c_\ell \chi_{[a_\ell,b_\ell)}$, where $[a_\ell,b_\ell)$ are disjoint and nonempty intervals, and $0=a_1<a_2<\ldots<a_L<b_L<2\pi=:a_{L+1}$. Here $c_\ell$'s are nonzero complex numbers and they are distinct if the corresponding intervals have a common endpoint.

It suffices to prove that for a positive density subset $S$ of $\N$ we have,
\be \label{setS}
\forall k\in S,\,\,\,\,\,\,\,\big|\widehat{|e^{it\partial_{xx}} g|^2}(k)\big|\gtrsim \frac1k.
\ee
First note that for $n\neq 0$
$$
\widehat{e^{it\partial_{xx}} g}(n)=\frac{i}{2\pi} \sum_{\ell=1}^{L} c_\ell  e^{-itn^2} \frac{e^{-inb_\ell}-e^{-ina_\ell}}{n}.
$$
Let $K$ be a natural number such that $Ka_\ell=Kb_\ell =0$ (mod $2\pi$) for each $\ell$. For $k$ divisible by $K$,
we have (using $\widehat{g}(k)=0$)
\begin{align*}
\widehat{|e^{it\partial_{xx}} g|^2}(k)&=\frac{1}{4\pi^2} \sum_{\ell,m=1}^{L} c_\ell \overline{c_m} \sum_{n\neq 0,k} e^{-itn^2}e^{it(n-k)^2} \frac{(e^{-inb_\ell}-e^{-ina_\ell}) (e^{inb_m}-e^{in a_m})}{n (n-k)}\\
&= \frac{e^{itk^2}}{4\pi^2 k} \sum_{\ell,m=1}^{L} c_\ell \overline{c_m} \sum_{n\neq 0,k} \frac{e^{-2itnk}}{n-k} (e^{-inb_\ell}-e^{-ina_\ell}) (e^{inb_m}-e^{in a_m})   \\
&-\frac{e^{itk^2}}{4\pi^2 k} \sum_{\ell,m=1}^{L} c_\ell \overline{c_m} \sum_{n\neq 0,k} \frac{e^{-2itnk}}{n} (e^{-inb_\ell}-e^{-ina_\ell}) (e^{inb_m}-e^{in a_m}).
\end{align*}
Changing the variable $n-k\to n$ in the first sum, we obtain
$$
\widehat{|e^{it\partial_{xx}} g|^2}(k)=O(1/k^2)-\frac{i\sin(k^2t)}{2\pi^2 k}\sum_{\ell,m=1}^{L} c_\ell \overline{c_m} \sum_{n\neq 0} \frac{e^{-2itnk}}{n} (e^{-inb_\ell}-e^{-ina_\ell}) (e^{inb_m}-e^{in a_m}).
$$
Using the formula (for $0<\alpha<2\pi$)
$$
\sum_{n\neq 0}\frac{e^{in\alpha}}{n}= \log(1-e^{-i\alpha})-\log(1-e^{i\alpha})=i(\pi-\alpha),
$$
we have
\begin{multline*}
\widehat{|e^{it\partial_{xx}} g|^2}(k)=O(1/k^2) \\ -\frac{\sin(k^2t)}{2\pi^2 k}\sum_{\ell,m=1}^{L} c_\ell \overline{c_m}
\big( \lfloor -2kt+b_m-b_\ell \rfloor  - \lfloor   -2kt+b_m-a_\ell\rfloor  - \lfloor  -2kt+a_m -b_\ell\rfloor +  \lfloor -2kt+a_m-a_\ell \rfloor  \big),
\end{multline*}
where $\lfloor x \rfloor = x$ (mod $2\pi) \in [0,2\pi)$. Note that for  $0\leq a<b\leq 2\pi$, we have
$$\lfloor c-b\rfloor - \lfloor c-a\rfloor = a-b+2\pi \chi_{[a,b)}(\lfloor c\rfloor).$$
Using this we have
\begin{multline*}
\widehat{|e^{it\partial_{xx}} g|^2}(k)=O(1/k^2) \\ -\frac{\sin(k^2t)}{ \pi  k}\sum_{\ell,m=1}^{L} c_\ell \overline{c_m}
\big( \chi_{[a_\ell,b_\ell)}(\lfloor -2kt+b_m \rfloor)  - \chi_{[a_\ell,b_\ell)}(\lfloor -2kt+a_m \rfloor)   \big).
\end{multline*}
From now on we restrict ourself to $k$ so that $\lfloor -2kt \rfloor \in(0,\epsilon)$ for some fixed
$$0<\epsilon<\min\big(\min_{\ell\in\{1,\ldots,L\}}(b_\ell-a_\ell),\min_{\ell\in\{1,\ldots,L\}:a_{\ell+1}\neq b_\ell}(a_{\ell+1}-b_\ell)\big),$$
where $a_{L+1}=2\pi$. We note that for such $k$, $\chi_{[a_\ell,b_\ell)}(\lfloor -2kt+a_m \rfloor)  =1$  for $m=\ell$ and it is zero otherwise.
Moreover, $\chi_{[a_\ell,b_\ell)}(\lfloor -2kt+b_m \rfloor)=0$ for each $m\neq \ell-1$ (mod $L$), and if $\chi_{[a_\ell,b_\ell)}(\lfloor -2kt+b_{\ell-1}) \rfloor)=1$, then $c_{\ell-1}\neq c_\ell$.

 Therefore,
\begin{multline*}
\widehat{|e^{it\partial_{xx}} g|^2}(k)=O(1/k^2) +\frac{\sin(k^2t)}{ \pi  k}\Big[\sum_{\ell =1}^{L} |c_\ell|^2 - \sum_{\ell =1}^{L} c_\ell \overline{c_{\ell-1}}  \chi_{[a_\ell,b_\ell)}(\lfloor -2kt+b_{\ell-1} \rfloor) \Big]\\
=O(1/k^2) +\frac{\sin(k^2t)}{ \pi  k}\Big[\sum_{\ell =1}^{L} |c_\ell|^2 - \sum_{\ell =1}^{L} c_\ell \overline{c_{\ell-1}} \delta_{a_\ell,b_{\ell-1}}   \Big].
\end{multline*}
Note that, since $c_\ell$'s are nonzero, the absolute value of the quantity in bracket is nonzero if $\delta_{a_\ell,b_{\ell-1}} =0$ 
for some $\ell$. In the case $\delta_{a_\ell,b_{\ell-1}} =1$ for each $\ell$, we write it as
$$
\sum_{\ell =1}^{L} |c_\ell|^2 - \sum_{\ell =1}^{L} c_\ell \overline{c_{\ell-1}} =C \cdot C- C\cdot \tilde C, 
$$
where $C=(c_1,\ldots, c_L)$ and $\tilde C= (c_L,c_1,\ldots,c_{L-1})$.
Since, in this case,  adjacent $c_{\ell}$'s are distinct, we have $\|\tilde C\|=\|C\|$ and $\tilde C\neq C$. Therefore, the absolute value of the quantity in bracket is nonzero.

Thus, \eqref{setS} holds for
$$
S=\big\{k\in\N: K|k , \lfloor -2kt \rfloor \in(0,\epsilon), \lfloor k^2t \rfloor \in [\pi/4,3\pi/4]\big\}.
$$
This set has positive density for any irrational $\frac{t}{2\pi}$ since
$\big\{\big( \lfloor-2Kjt \rfloor, \lfloor K^2j^2t \rfloor\big):j\in\N\big\}$ is uniformly distributed on $\T^2$ by the classical Weyl's theorem (see e.g. Theorem 6.4 on page 49 in \cite{kn}).
\end{proof}

We now consider the case of higher order dispersion. 
\begin{theorem}\label{thm:higherdispersion}
Fix an integer $k\geq 3$.   Let  $g:\T \to \C $ be of bounded variation. Then $e^{it (-i \partial_x)^k}g$ is a continuous function of $x$ for almost every $t$. Moreover if in addition  $g\not \in \bigcup_{\epsilon>0}H^{\frac12+\epsilon}$, then for almost all $t$ both  the real part and the imaginary part of the graph of
$e^{it (-i \partial_x)^k}g $ have upper Minkowski dimension $D\in [1+2^{1-k},2-2^{1-k}]$.  
\end{theorem}

In particular, under the conditions of the theorem, the solution of the Airy equation, 
$$u_t+u_{xxx}=0,\,\,\,\,u(0,x)=g(x),\,\,x\in\T,$$
has upper Minkowski dimension $D\in [\frac54,\frac74]$ for almost every $t$.
 
The proof of this theorem is similar to the proof of the $k=2$ case above, by replacing Theorem~\ref{weyl} with
\begin{theorem}\label{weylk}\cite{montgomery} Fix an integer $k\geq 2$. Let $t$ satisfy \eqref{dirichlet} for some integers $a$ and $q$,
then
$$
\sup_x \big|\sum_{n=1}^N e^{itn^k+inx}\big| \lesssim N^{1+\epsilon}\big(\frac1q+\frac1N+\frac{q}{N^k} \big)^{2^{1-k} }.  
$$
\end{theorem}
Note that in this case,
$$
e^{it (-i \partial_x)^k}g \in \Big[\bigcap_{\epsilon>0} B_{\infty,\infty}^{2^{1-k}-\epsilon}(\T)\Big] \bigcap C^0(\T).
$$
Combining the results above with the smoothing theorem of \cite{erdtzi}, we have the following:
\begin{cor}
Let  $g:\T \to \C $ be of bounded variation. Then  the solution of KdV with data $g$ is a continuous function of $x$ for almost every $t$. Moreover if in addition  $g\not \in \bigcup_{\epsilon>0}H^{\frac12+\epsilon}$, then for almost all $t$   the graph of
the solution has upper Minkowski dimension $D\in [\frac54,\frac74]$.  
\end{cor}

\section{Fractal solutions of the periodic vortex filament equation}

In this section we will construct  solutions of VFE with some fractal behavior via the SM equation \eqref{sm}.
For the existence and uniqueness of smooth solutions of SM see the discussion in \cite{JS}.

Let $u:\T\to S^2$ be a smooth map and denote by  $\Gamma(u)_x^y$ the parallel transport along $u$ between the points $u(x)$ and $u(y)$. We say $u$ is identity holonomy if $\Gamma(u)_0^{2\pi}$ is the identity map on the tangent space $T_{u(0)}S^2$. For $u\in H^s$ for some $s\geq 1$, we say $u$ is identity holonomy if there is a sequence of smooth and identity holonomy maps $u_n$ converging to $u$ in $H^s$.
Recall that for smooth planar initial curves for VFE, the data for SM is identity holonomy and mean zero.

\begin{theorem}
\label{hs32}
The initial value problem for the Schr\"odinger map equation:
\begin{align*}
&u_t=u\times u_{xx}, \,\,\,\,\,x\in\T, t\in \R,\\
&u(x,0)=u_0(x) \in H^s(\T),
\end{align*}
is globally well-posed for $s>\frac32$ whenever $u_0$ is identity holonomy and mean-zero. Moreover, there is a unique (up to a rotation) frame $\{e,u\times e\}$, with $e\in H^s$, and a complex valued function (unique up to a modulation), $q_0\in H^{s-1}(\T)$, so that the evolution of the curvature vector $\gamma_{xx}=u_x$ satisfies
$$
u_x= q_{1} \, e  + q_{2 } \, u\times e,
$$
where $q=q_1+iq_2$ solves NLS with data $q_0$. In particular, the curvature $|u_x|$ is given by $|q|$.
\end{theorem}

We start by transforming the Schr\"odinger map equation to a system of ODEs following \cite{csu} and \cite{nsvz}. Let $u_0:\T\to S^2$ be smooth, mean-zero and identity holonomy. We pick a unit vector $e_0(0)\in T_{u_0(0)}S^2$, and define $e_0:\T \to S^2 $ by parallel transport
\begin{equation}
\label{e0}
e_0(x)=\Gamma(u_0)_0^x [e_0(0)], \, x\in \T. 
\end{equation}
Notice that $e_0$ is $2\pi$-periodic since $u_0$ is identity holonomy.
We remark that for each $x$, $\{e_0(x), u_0(x)\times e_0(x)\}$ is an orthonormal basis for $T_{u_0(x)}S^2$, and $\{u_0(x), e_0(x), u_0(x)\times e_0(x)\}$ is an orthonormal basis for $\R^3$.
Therefore, we can write  $\partial_x u_0(x)$ in this frame as
$$
\partial_x u_0(x) = q_{1,0}(x)\, e_0(x) + q_{2,0}(x)\, u_0(x)\times e_0(x).
$$
Noting that $\partial_x e_0(x)$ is a scalar multiple of $u_0(x)$, and using the identity
$$
0=\partial_x [u_0(x)\cdot e_0(x)]=[\partial_x  u_0(x)] \cdot e_0(x) + u_0(x)\cdot \partial_x [ e_0(x)]=q_{1,0}+u_0(x)\cdot \partial_x [ e_0(x)],
$$
we write
$$
\partial_x e_0(x)=-q_{1,0}(x) u_0(x).
$$   
We remark that both $q_{1,0}$ and $q_{2,0}$ are $2\pi$-periodic and we define $q_0:=q_{1,0}+iq_{2,0}:\T\to \C$.

Let $q:\T\times \R \to \C$ be the smooth solution of the focusing cubic NLS equation
\begin{align}\label{nls}
iq_t+q_{xx}+\frac12|q|^2q=0,
\end{align}
with initial data $q(x,0)=q_0(x).$ We also define 
\begin{equation}\label{pq}
p=p_1+ip_2:=i\partial_x q.
\end{equation}
The following theorem exemplifies the connection between NLS and VFE. We note again that   the coordinates of the curvature vector $\gamma_{xx}=u_x$ in the frame $\{e,u\times e\}$ is given by the real and imaginary parts of the solution of NLS.
\begin{theorem}\label{system}
Let $u_0, e_0, q,$ and $p$  be as above.  Let $u,e$ solve the system of ODE
\begin{align}
\label{ut} &\partial_t u = p_1 e +p_2 u\times e  \\
\label{et} &\partial_t e =-p_1 u -\frac12 |q|^2 u\times e\\
&e(x,0)=e_0(x),\,\,\,u(x,0)=u_0(x). \nn
\end{align} 
Then we have  $ u:\T\times \R \to S^2$, $e$ is $2\pi$-periodic in 
$x$, and for each $x,t$,  $e(x,t)\in T_{u(x,t)}S^2$,   $\|e(x,t)\|=1$, and  
\begin{align}
\label{ux} &\partial_x u=  q_{1} e  + q_{2 } \, u\times e\\
\label{ex} &\partial_x e= -q_1 u.
\end{align}
Moreover, $u$ is the unique $2\pi$-periodic smooth solution of the Schr\"odinger map equation 
\begin{align}\label{sm2}
u_t=u\times u_{xx},
\end{align}
with initial data $u(x,0)=u_0(x)$.  
\end{theorem}
\begin{proof}
First we check the last assertion of the theorem by showing that $ u$ solves \eqref{sm2}. Using \eqref{ux} and \eqref{ex}, we have
\begin{align*}
&u\times u_{xx}=(q_1)_x u\times e+q_1 u\times e_x + (q_2)_x u\times (u\times e) + q_2 u\times (u_x\times e) + q_2 u\times (u\times e_x)\\
& =p_2 u\times e+p_1 e =u_t.
\end{align*}
Next note that $u\cdot u=e\cdot e=1$, and $u\cdot e=0$ for all times since these quantities initially take these values and they solve the linear system
\begin{align*}
&\partial_t (e\cdot e) = -2 p_1 \, u\cdot e\\
&\partial_t (u\cdot u) = 2p_1 \,u \cdot e\\
&\partial_t (e\cdot u)= p_1 \, (e\cdot e -u\cdot u). 
\end{align*}
Therefore the local solutions of the system of ODE \eqref{ut}, \eqref{et} are $S^2$ valued, and hence they extend globally in time. Moreover, since they satisfy $e \cdot u=0$, $e(x,t)\in T_{u(x,t)}S^2$ for all $x, t$. 

\vskip 0.05in
\noindent
We   now prove that the identities \eqref{ux} and \eqref{ex} hold true for all times. Let 
$$f(x,t):=\partial_x u- q_{1} e  - q_{2 } \, u\times e,$$
$$g(x,t):=\partial_x e+q_1 u.$$
We compute using \eqref{ut} and \eqref{et}
\begin{align*}
&g_{t} =\partial_{x}e_{t}+(q_1)_t u+q_{1}u_t\\
&\quad=\partial_{x}\Big ( -p_1u-\frac{1}{2}|q|^2u\times e\Big)+(q_1)_t u+q_1p_1e+q_1p_2u\times e\\
&\quad =-(p_1)_{x}u-p_1 u_x-\frac{1}{2}(|q|^2)_{x}u\times e\\
&\quad \quad - \frac{1}{2}|q|^2\,u_x\times e-\frac{1}{2}|q|^2u\times e_x+(q_1)_tu+q_1p_1e+q_1p_2u\times e.
\end{align*}
Using the definitions of $f$ and $g$, we obtain
\begin{multline*}
g_t=\Big( -\frac{1}{2}(|q|^2)_x+q_1p_2-p_1q_2\Big)\, u\times e+\Big(-(p_1)_x+(q_1)_t+\frac{1}{2}|q|^2q_2\Big)\, u \\ -p_1f -\frac12|q|^2 \big(f\times e +u\times g).
\end{multline*}
Using the NLS equation \eqref{nls} and the relation \eqref{pq} between $p$ and $q$,  we have that
\begin{align*}
g_t= -p_1f -\frac12|q|^2 \big(f\times e +u\times g).
\end{align*}
By performing analogous calculations one obtains
\begin{align*}
f_t=  p_1g + (p_2-q_2) \big(f\times e +u\times g).
\end{align*}
Since $f$ and $g$ are initially zero,  \eqref{ux} and \eqref{ex} hold true for all times.
\end{proof}

\begin{remark} We note that $q$ and $e$ in the system above depend on the choice of $e_0(0)$, however, because  of the uniqueness of smooth solutions, $u$ is independent of this choice. Moreover, if we pick  $\tilde{e_0}(0)=e^{i\alpha} e_0(0)$, then by the properties of parallel transport, $\tilde{e_0}(x)=e^{i\alpha} e_0(x)$ and $\tilde{e_0}(x) \times u_0(x)=e^{i\alpha} [ e_0 (x) \times u_0(x)]$ for each $x$. And hence, in this new frame, we have $\tilde {q_0}(x) =e^{-i\alpha} q_0(x)$. 
By the properties of NLS evolution and \cite{csu} (or the system above) for each $t$, we have $\tilde {q}(t) =e^{-i\alpha} q(t)$ and $\tilde{e}(x,t)=e^{i\alpha} e(x,t)$.
 \end{remark}
The following lemma, which will be proved in the appendix, relates the $H^s$ norms of the parallel transform and of the curve. 
\begin{lemma}\label{lem:pt} Fix $s\geq 1$.
Let $u,v:\T\to S^2$ be smooth and identity holonomy functions. Pick unit vectors $e(0)\in T_{u(0)}S^2$ and $f(0)\in T_{v(0)}S^2$ so that 
$$|e(0)-f(0)|\les \|u-v\|_{H^s}.$$ Let $e(x)=\Gamma(u)_0^{x}e(0)$ and $f(x)=\Gamma(v)_0^{x}e(0)$. Then\\
i) $\|e\|_{H^s(\T)} \leq C_{\|u\|_{H^s}}$, and\\
ii) $\|e-f\|_{H^s(\T)} \leq C_{\|u\|_{H^s}, \|v\|_{H^s}} \|u-v\|_{H^s}$.
\end{lemma}
 
To construct the local and global solutions of SM we need the following a priori bounds:
\begin{lemma}\label{lem:nb}
Let $u,e,q$ be as in Theorem~\ref{system}. Then for any $s\geq 1$, and for each  $T$, we have
$$
\sup_{t\in (-T,T)} \big(\|e(t)\|_{H^s}+\|u(t)\|_{H^s}+\|q(t)\|_{H^{s-1}}\big)\leq C_{T,\|u_0\|_{H^s}}
$$
\end{lemma}
\begin{proof}
First recall the conservation law 
$$
\|u\|_{H^1} =\|u_0\|_{H^1}.
$$
By \eqref{ux}, \eqref{ex}, we have
\begin{align}\label{q1q2}
q_1=-e_x\cdot u,\,\,\,\, q_2=u_x\cdot (u\times e).
\end{align}
Therefore, by the algebra property of $H^{s-1}$ (or by Lemma~\ref{frac} if $1\leq s\leq 3/2$) and Lemma~\ref{lem:pt}, we have
$$
\|q_0\|_{H^{s-1}}\les C_{\|u_{0}\|_{H^s}}.
$$
By NLS theory \cite{bourgain} we have for $t\in[-T,T]$
\begin{align}\label{qbound}
\|q\|_{H^{s-1}} \leq C_{T,\|u_{0}\|_{H^s}}.
\end{align}
Combining these with \eqref{ex} we see that for $t\in[-T,T]$
\begin{align} \label{ebound}
\|e\|_{H^s} \les 1+\|\partial_x e\|_{H^{s-1}} \les 1+  \|q\|_{H^{s-1}} \|u\|_{H^1} \leq C_{T,\|u_{0}\|_{H^s}}.
\end{align}
Similarly, using \eqref{ux}, we have for $t\in[-T,T]$
\begin{align} \label{ubound}
\|u\|_{H^s} =\|\partial_x u \|_{H^{s-1}} \leq C \|q\|_{H^{s-1}} \|e\|_{H^{1}} \|u\|_{H^1}  \leq C_{T,\|u_{0}\|_{H^s}}.
\end{align}
\end{proof}

Using the following lemma and the proposition we  construct the global solution of SM as a limit of a uniformly Cauchy sequence. 

\begin{lemma}\label{difference}
Fix $s\in (\frac32,2]$, and $T>0$.
Let $(u,e,q_u)$ and $(v,f,q_v)$ be as in Theorem~\ref{system} satisfying
$$ |e_0(0)-f_0(0)|\les \|u_0-v_0\|_{H^s}. $$
Then for $t\leq T$
$$
\|e-f\|_{\dot H^{s}} \lesssim \|u_0-v_0\|_{H^s} + \|u-v\|_{H^{s-1}},\,\,\,\,\,\, \|u-v\|_{H^{s}} \lesssim \|u_0-v_0\|_{H^s}+\|u-v\|_{H^{s-1}}+  |P_0(e-f)|,
$$
and
$$
\|e-f\|_{\dot H^{s-1}} \lesssim \|u_0-v_0\|_{H^s} + \|u-v\|_{H^{s-2}},\,\,\,\,\,\, \|u-v\|_{H^{s-1}} \lesssim \|u_0-v_0\|_{H^s}+\|u-v\|_{H^{s-2}}+  |P_0(e-f)|,
$$
with   implicit constants depending on $T, \|u_0\|_{H^s}, \|v_0\|_{H^s}$.
\end{lemma}
\begin{proof}
First note that by using  Lemma~\ref{lem:pt},   and the identities \eqref{q1q2}, we  get
$$
\|q_u(0)-q_v(0)\|_{H^{s-1}}\leq C_{\|u_0\|_{H^s},\|v_0\|_{H^s} }   \|u_{0}-v_{0}\|_{H^s}.
$$
By Lipschitz dependence on initial data for the  NLS evolution, we have
\begin{align}\label{qcauchy}
\|q_{u}-q_{v}\|_{L^\infty_{[-T,T]} H^{s-1}_x}\leq C_{T,\|u_0\|_{H^s},\|v_0\|_{H^s}}  \|u_{0}-v_{0}\|_{H^s}.
\end{align}

By using \eqref{ex} and Lemma~\ref{frac} we have 
\begin{align}\nn
\|e-f \|_{\dot H^{s-1}}&= \|\partial_xe-\partial_x f\|_{H^{s-2}} = \|q_{v,1} v -q_{u,1} u\|_{H^{s-2}}\\
&\leq \|(q_{v,1}   -q_{u,1}) u\|_{H^{s-2}}+\| q_{u,1} (u-v)\|_{H^{s-2}} \nn \\
&\les  \|q_{v }   -q_{u }\|_{H^{s-1}} \| u\|_{H^{s}}+\| q_{u }\|_{H^{s-1}}\| u-v\|_{H^{s-2}}. \nn
\end{align}
The statement for $\|e-f\|_{\dot H^{s-1}}$  follows by \eqref{qcauchy}, and Lemma~\ref{lem:nb}.

Similarly, by \eqref{ux} and Lemma~\ref{frac}, we obtain
\begin{align}\nn
\|u-v \|_{  H^{s-1}}&= \|\partial_xu-\partial_x v\|_{H^{s-2}} \leq \|q_{u,1} e -q_{v,1} f\|_{H^{s-2}}+\|q_{u,2} u\times e -q_{v,2} v\times f\|_{H^{s-2}}\\ 
&\les  \|q_{u }   -q_{v }\|_{H^{s-1}}  + \|e-f\|_{H^{s-2}} +\| u-v\|_{H^{s-2}},\nn
\end{align}
which yields the bound for $\|u-v\|_{H^{s-1}}$.

The bounds for $H^{s}$ norms follow  from the calculations above  by using the algebra structure of Sobolev spaces    instead of  Lemma~\ref{lem:nb}.
\end{proof}

\begin{proposition}\label{boot}
Fix $s\in (\frac32,2]$, and $T>0$.
Let $(u,e,q_u)$ and $(v,f,q_v)$ be as in Theorem~\ref{system} satisfying
$$ |e_0(0)-f_0(0)|\les \|u_0-v_0\|_{H^s}. $$
Then for $t\leq T$ we have
$$
\|u-v\|_{H^{s}} \lesssim \|u_0-v_0\|_{H^s},\,\,\,\,\,\|e-f\|_{H^{s}} \lesssim \|u_0-v_0\|_{H^s}
$$
with an implicit constant depending on $T, \|u_0\|_{H^s}, \|v_0\|_{H^s}$.
\end{proposition}
\begin{proof}
First note that \eqref{qcauchy} is valid.

Let $k^2(t):=|P_0(e-f)|^2+ \|u-v\|_{H^{s-2}}^2$.
By Lemma~\ref{difference}, it suffices to prove that  $k(t)\les \|u_0-v_0\|_{H^s}$ which follows from 
\begin{align}\label{kbound}
k^\prime(t)\les \|u_0-v_0\|_{H^s} +k(t).
\end{align}

We write
$$
\partial_t\big(\|u-v\|_{H^{s-2}}^2\big)=2\int_\T \partial_x^{s-2} (u-v) \cdot \partial_x^{s-2}(u_t-v_t) \, dx.
$$
Using \eqref{ut} and Cauchy--Schwarz we have
\begin{align}\label{H-1t}
\partial_t\big(\|u-v\|_{H^{s-2}}^2\big)\les \|u-v\|_{H^{s-2}} \big(\|p_{u,1}e-p_{v,1}f\|_{H^{s-2}}+\|p_{u,2} u\times e - p_{v,2} v\times f\|_{H^{s-2}} \big).
\end{align}
To estimate $\|p_{u,1}e-p_{v,1}f\|_{H^{s-2}}$, we write using \eqref{pq} and Lemma~\ref{frac}
\begin{align*}
\|p_{u,1}e-p_{v,1}f\|_{H^{s-2}}& \leq \|[(q_{v,2})_x-(q_{u,2})_x] e\|_{H^{s-2}} +\|(q_{v,2})_x (f-e)\|_{H^{s-2}}  \\
&\les \| (q_{v,2})_x-(q_{u,2})_x \|_{H^{s-2}}  \|e\|_{H^{s-1}} +\|(q_{v,2})_x\|_{H^{s-2}}  \|f-e \|_{H^{s-1}}\\
&\les  \|  q_{v }  - q_{u}  \|_{H^{s-1}} +  \| q_{v } \|_{H^{s-1}} \|f-e \|_{H^{s-1}}.
\end{align*}
Now by using Lemma~\ref{lem:nb}, Lemma~\ref{difference}, and \eqref{qcauchy}, we have   
$$
\|p_{u,1}e-p_{v,1}f\|_{H^{s-2}}\les \|u_0-v_0\|_{H^s}+ k(t),
$$
where the implicit constants depend on  $ \| u_0\|_{H^{s}}$, $\| v_0\|_{H^{s}}$, and $T$.

Similarly one estimates 
$$
\|p_{u,2} u\times e - p_{v,2} v\times f\|_{H^{s-2}} \les  \|u_0-v_0\|_{H^s}+ k(t).
$$
Thus, we get
$$
\partial_t\big(\|u-v\|_{H^{s-2}}^2\big)\les  \|u_0-v_0\|_{H^s} k(t) + k^2(t).
$$

Now note that by \eqref{et}, we have
$$
\partial_t |P_0(e-f)|^2 \les |P_0(e-f)|  \Big( \Big| \int_\T \big( p_{u,1}u-p_{v,1}v \big)dx \Big| +\Big|\int_\T \big(|q_u|^2u\times e -|q_v|^2v\times f\big) dx \Big| \Big).
$$
We estimate the first integral on the right hand side above by using \eqref{pq} and integration by parts as follows:
\begin{align*}
\Big| \int_\T \big( p_{u,1}u-p_{v,1}v \big)dx \Big| & = \Big| \int_\T \big( q_{u,2}u_x-q_{v,2}v_x \big)dx \Big| \\
&\les \Big| \int_\T \big( q_{u,2}q_{u,1} e- q_{v,2}q_{v,1} f  \big)dx \Big| + \Big| \int_\T \big( q_{u,2}^2 u\times  e- q_{v,2}^2 v\times f  \big)dx \Big|\\
&\les \|q_u-q_v\|_{H^{s-1}}+\|u-v\|_{H^{s-1}}+\|e-f\|_{H^{s-1}} \\ &\les \|u_0-v_0\|_{H^s}+k(t).
\end{align*}
In the first inequality we used \eqref{ux}, and in the second we bound the differences by Sobolev embedding. The last inequality follows from \eqref{qcauchy} and Lemma~\ref{difference}.

The estimate for the second integral is similar, and hence \eqref{kbound} holds.
\end{proof}

We now finish the proof of Theorem~\ref{hs32}. Given initial data $u_0$, choose a sequence  $u_{0,n}$ of smooth and identity holonomy functions converging to $u_0$ in the  $H^s$ norm. For each $u_{0,n}$ we choose $e_{0,n}(0)$ so that (for each $n,m$)
$$
|e_{0,n}(0)-e_{0,m}(0)|\lesssim \|u_{0,n}-u_{0,m}\|_{H^s}.
$$
Consider the system given by Theorem~\ref{system} for each $n$. By Proposition~\ref{boot} we see that 
for each $T$, $u_n$ and $e_n$ are Cauchy in $C^0_{[-T,T]}H^{s}$, and hence they  converge to functions $u$ and $e$ in $C^0_{[-T,T]}H^{s}$. This implies existence and uniqueness. Continuous dependence on initial data also follows immediately from Proposition~\ref{boot}. The proof of the claim on the curvature is clear from the construction above.

\section{Weak Solutions of SM on the torus.}
The construction of weak solutions in the energy space was proved in \cite{SSB}, also see the discussion in \cite{JS}.  
Certain uniqueness statements (under assumptions on NLS evolution on $\R$) were obtained in \cite{nsvz}. In this section we obtain unique weak solutions in $H^s(\T)$ for $s\geq 1$ and for identity holonomy and mean zero data. These solutions  are weakly continuous in $H^s$ and continuous in $H^r$ for $r<s$. Moreover, for $s>1$ the curvature $u_x$ is given by a NLS evolution in $H^{s-1}$ level as in the previous section.

First we  discuss weak solutions in $H^1$ for  identity holonomy and mean zero data $u_0\in H^1$. Take a smooth identity holonomy sequence $u_n(0)$ converging to $u_0$ in $H^1$. Construct solutions $u_n,e_n,q_n$ as in the previous section. 
By Lemma~\ref{lem:nb},  we have for each $T$
$$
\sup_{n,|t|\leq T} \big(\|u_n\|_{H^1}+\|e_n\|_{H^1}+\|q_n\|_{L^2}\big) \leq C_{T, \|u_0\|_{H^1}}.
$$
Moreover, by the equation,
$$
\sup_{n,|t|\leq T} \big(\|\partial_t u_n\|_{H^{-1}}+\|\partial_t e_n\|_{H^{-1}} \big) \leq C_{T, \|u_0\|_{H^1}}.
$$
Having these bounds we   apply Proposition 1.1.2 in \cite{caz} to construct a weak solution.  Taking  $X=H^1$, $Y=H^{r}$ for any $r<1$, we estimate
$$
\|u_n(t_1)-u_n(t_2)\|_{H^{-1}}\leq \int_{t_2}^{t_1} \|\partial_tu_n(t)\|_{H^{-1}} dt \leq  C_{T, \|u_0\|_{H^1}} |t_1-t_2|.
$$
Interpolating this inequality with the $H^1$ bound gives equicontinuity in $H^{r}$ for $r<1$. Similar statements hold for $e_n$'s.

The proposition yields $u\in C^0_tH^{r}$ which is weakly continuous in $H^1$ and a subsequence $n_k$ such that
$u_{n_k}$ converges to $u$ weakly in $H^1$ for each $t$, and same for $e$. By passing to further subsequences
we also have
$\partial_t u_{n_k}$ converges to $\partial_t u$ weakly in $L^2_tH^{-1}_x$, same for $\partial_t e$. The limit satisfies the system and the Schr\"odinger map equation in the sense of space-time distributions.

We further note that for each $t$, by Rellich and the uniqueness of weak limits, a subsequence $u_{n_k}$ (depending on $t$) converges to $u$ strongly in $H^{r}$. Therefore, (by Sobolev embedding)
$|u(x,t)|=1$ and similarly $|e(x,t)|=1$ and $e(x,t)\cdot u(x,t)=0$ for each $t,x$.

For the uniqueness part we use an argument from \cite{nsvz}. Given two such solutions $u$ and $v$, consider the smooth solutions $u_n$ and $v_n$ converging to $u$ and $v$ as above. 
In particular, $u_n(0)$ and $v_n(0)$ converges to the same data in $H^1$. Since the solutions are independent of the choice of the frame, we can choose $e_{0,n}(0)$ and $f_{0,n}(0)$ so that in addition to the conditions above we have 
$$
|e_{0,n}(0)-f_{0,n}(0)|\les \|u_n(0)-v_n(0)\|_{H^1}.
$$
With this choice, we consider the two weak solutions $(u,e,q_u), (v,f,q_v)$ of the system in Theorem~\ref{hs32}. Note that $u_0=v_0$, $e_0=f_0$, and $q_u(0)=q_v(0)$. By the well-posedness of NLS
$q_u=q_v=q$ for all times.    Let $U=u-v$, $E=e-f$,   $C=u\times e -v\times f$, and $V=(U,E,C)$.   
Note that $\partial_t V=BV$ (in the sense of space-time distributions), where
\begin{equation*}
B =\left[\begin{array}{ccc} 0& p_{u,1}& p_{u,2}\\  -p_{u,1}&0&-\frac12|q_u|^2\\-p_{u,2}&\frac12|q_u|^2&0
\end{array}\right] 
\end{equation*} 

Using $H^1$ and $H^{-1}$ duality, and the fact that $B$ is skew symmetric, we see that 
$$
\partial_t \|V\|_2^2= \int BV\cdot V =0.
$$
 Since $V(0)=0$, we have uniqueness.

Finally, we discuss the weak solutions in $H^s$ for $s>1$. 
By the argument above (and Proposition 1.1.2 in \cite{caz}), the weak solution is unique, it is  in $C^0_tH^1_x$, and weakly continuous in time with values in $H^s$.
In fact in this case the solution enjoys conservation of $H^1$ norm and it is the strong limit of $u_{n_k}$ in $H^1$.  
Indeed, by Rellich and the uniqueness of weak limits, at each time,  a $t$ dependent subsequence $u_{n_k}(t)$ converges to $u(t)$ in $H^1$. Therefore
$$\|u(t)\|_{H^1}=\lim_{k\to\infty} \|u_{n_k}(t)\|_{H^1} = \lim_{k\to\infty} \|u_{n_k}(0)\|_{H^1} =\|u_0\|_{H^1}.$$
Hence the full sequence $\|u_{n_k}\|_{H^1}$ converges to $\|u\|_{H^1}$ for each $t$. This also implies by applying  Proposition 1.1.2 in \cite{caz} one more time with $B=H^1$ that 
$u_{n_k}$ converges to $u$ strongly in $H^1$ for each $t$. In particular, $|\partial_x u_{n_k}(t)|$ converges to $|\partial_x u(t)|$ in $L^2$.
Also note that by the system above and by the NLS theory $|q_{n_k}(t)|=|\partial_x u_{n_k}(t)|$ converges to $|q(t)|$ in $L^2$. Therefore for each $t$, $|q(t)|=|u_x(t)|$ as $L^2$ functions.  Similarly, we have $u_x=q_1\, e+q_2\,u\times e$, for each $t$.

\section{Appendix}
\begin{lemma}\label{frac}
For $\alpha\in [-\frac12,\frac12]$, we have
$$
\|fg\|_{H^\alpha}\lesssim \|f\|_{H^{1/2+}} \|g\|_{H^\alpha}.
$$
\end{lemma}
\begin{proof}
First note that for $\alpha=0$ the lemma follows from Sobolev embedding theorem. Moreover if one proves the estimate for any $0<\alpha\leq \frac{1}{2}$ then the estimate for $-\frac{1}{2}\leq\alpha<0$ follows by duality. Indeed consider $h\in H^{-\alpha}$ and estimate
$$\int_{\Bbb T} fgh =\int_{\Bbb T} g (fh) \lesssim \|g\|_{H^{\alpha}}\|fh\|_{H^{-\alpha}}\lesssim \|g\|_{H^{\alpha}}\|f\|_{H^{\frac{1}{2}+}}\|h\|_{H^{-\alpha}}.$$ 
We only show the calculation for $\alpha=\frac12$, the middle range follows by interpolation.  

$$\|fg\|_{H^{\frac{1}{2}}}=\big\|\la n \ra^{\frac{1}{2}}\sum_{k}\widehat{f}(n-k)\widehat{g}(k)\big\|_{l^2}\leq \big\|\sum_{k}\la n \ra^{\frac{1}{2}}|\widehat{f}(n-k)|\ |\widehat{g}(k)|\ \big\|_{l^2}.$$
Now consider the $l^2$ functions defined by $u(k)=\widehat{f}(k)\la k\ra^{\frac{1}{2}+}$ and $v(k)=\widehat{g}(k)\la k\ra^{\frac{1}{2}}$. It suffices to show that 
$$\Big\|\sum_{k}\la n \ra^{\frac{1}{2}}\frac{u(n-k)}{\la n-k\ra^{\frac{1}{2}+}}\ \frac{v(k)}{\la k \ra^{\frac{1}{2}}}\ \Big\|_{l^2}\lesssim \|u\|_{l^2}\|v\|_{l^2}.$$
\\
{\it Case 1:} $|k|\lesssim |n-k|$. In this case $\la n \ra^{\frac{1}{2}}\lesssim \la n-k\ra^{\frac{1}{2}}$ and 
\\
$$\frac{\la n \ra^{\frac{1}{2}}}{\la n-k \ra^{\frac{1}{2}+}\la k \ra^{\frac{1}{2}}} \lesssim \frac{1}{\la n-k \ra^{+}\la k \ra^{\frac{1}{2}}}.$$
It follows that 
$$\Big\|\sum_{k}\la n \ra^{\frac{1}{2}}\frac{u(n-k)}{\la n-k\ra^{\frac{1}{2}+}}\ \frac{|v(k)|}{\la k \ra^{\frac{1}{2}}}\ \Big\|_{l^2}\lesssim \Big\|\frac{u}{\la \cdot \ra^{+}}* \frac{v}{\la \cdot \ra^{\frac{1}{2}}}\Big\|_{l^2} \lesssim \Big\|\frac{u}{\la \cdot \ra^{+}}\Big\|_{l^{2-}}\ \Big\|\frac{v}{\la \cdot \ra^{\frac{1}{2}}}\Big\|_{l^{1+}}\lesssim \|u\|_{l^2}\|v\|_{l^2}.$$
Note that in the second to last inequality we used Young's inequality and in the last inequality we used H\"older's inequality in $u$ and $v$ respectively.
\\
\\
{\it Case 2:} $|n-k|\lesssim |k|$. In this case $\la n \ra^{\frac{1}{2}}\lesssim \la k\ra^{\frac{1}{2}}$ and 
\\
$$\frac{\la n \ra^{\frac{1}{2}}}{\la n-k \ra^{\frac{1}{2}+}\la k \ra^{\frac{1}{2}}} \lesssim \frac{1}{\la n-k \ra^{\frac{1}{2}+}}.$$
It follows that 
$$\Big\|\sum_{k}\la n \ra^{\frac{1}{2}}\frac{u(n-k)}{\la n-k\ra^{\frac{1}{2}+}}\ \frac{v(k)}{\la k \ra^{\frac{1}{2}}}\ \Big\|_{l^2} \lesssim 
\Big\|\frac{u}{\la \cdot \ra^{\frac{1}{2}+}}* v\Big\|_{l^2}\lesssim \| u \, \la \cdot \ra^{-\frac12-}\|_{l^{1}}\ \|v\|_{l^{2}}\lesssim \|u\|_{l^2}\|v\|_{l^2}$$
by using Young's and H\"older's inequalities in that order.
\end{proof}

\begin{proof}[Proof of Lemma~\ref{lem:pt}] We will use the notation of \cite{bar}. By the definition of the parallel transport 
\begin{equation}
\label{pt1}
\partial_x e(x)= \big( \partial_x e(x) \cdot  u(x) \big)  u(x).
\end{equation} 
Fix a smooth local parametrization $(U,F, V)$ of $S^2$ such that $u(\T) \subset V$. Let  $\widetilde{u}:=F^{-1} \circ u: \T \rightarrow U$.
Write $e$ in the local parameters as 
\begin{equation}
\label{eu}
e(x)=\xi^1(x) D_1 F(\widetilde{u}(x))+\xi^2(x) D_2 F (\widetilde{u}(x)).
\end{equation}

We can write \eqref{pt1}  as,
\begin{equation}
\label{ptsys}
\partial_x{\xi}^k(x)=-\sum_{i,j=1}^2 \Gamma_{ij}^k (\widetilde{u}(x)) \, \partial_x\widetilde{u}_{j}(x)\xi^i (x),\, k=1,2, \, x\in \T,
\end{equation}
where $\Gamma_{ij}^k$ are the Christoffel symbols with respect to the local parametrization $(U,F,V)$.

Since $|e|=1$, we deduce that $\xi^1, \xi^2 \in L^\infty$, with a bound depending on $F$. Also note that, since $\Gamma_{ij}^k$ and $F^{-1}$ are smooth, we have
$$
\|\Gamma_{ij}^k (\widetilde{u}(x))\|_{H^s},\,\|D_jF (\widetilde{u}(x))\|_{H^s}\les 1+ \|\widetilde{u}\|_{H^s}\les \|u\|_{H^s}.
$$

Using this and Sobolev embedding in \eqref{ptsys}, we have
$$
\|\xi\|_{H^1}\les \|\xi\|_{L^\infty}+ \|\partial_x\xi\|_{L^2}\les  1+ \|u\|_{H^1}^2\|\xi\|_{L^\infty}\les \|u\|_{H^1}^2.
$$
And hence, for $s\in[1,2]$, we have
$$
\|\xi\|_{H^s}\les 1+\|\partial_x\xi\|_{H^{s-1}}\les 1+ \|u\|_{H^s}^2\|\xi\|_{H^1} \les \|u\|_{H^s}^4.
$$
Using this in \eqref{eu}, we obtain
$$
\|e\|_{H^s}\les \|u\|_{H^s}^5.
$$

For the second part, first write
$$
f(x)=\eta^1(x)D_1F(\widetilde {v}(x))+\eta^2(x)D_2F(\widetilde{v}(x)).
$$
Since 
$$
|e(0)-f(0)|\les \|u_0-v_0\|_{H^s}, \text{ and } |D_jF(\widetilde u (0)) - D_jF(\widetilde v (0))| \les \|u_0-v_0\|_{H^s},
$$
we have
$$
|\eta(0)-\xi(0)|\les  \|u_0-v_0\|_{H^s}.
$$
We also have
$$
\|\Gamma_{ij}^k (\widetilde{u}(x))-\Gamma_{ij}^k (\widetilde{v}(x))\|_{H^s},\,\|D_jF (\widetilde{u}(x))-D_jF (\widetilde{v}(x))\|_{H^s}\les \|u-v\|_{H^s}.
$$
Using these in \eqref{ptsys} as above, we obtain
$$
\|\eta-\xi\|_{H^s}\les \|u-v\|_{H^s},
$$
which implies that
$$
\|e-f\|_{H^s}\les \|u-v\|_{H^s}.
$$
\end{proof}

\end{document}